\documentclass[11pt]{amsart}
\usepackage{amssymb}
\usepackage{verbatim}

\newtheorem{lemma}{Lemma} [section]
\newtheorem{thm}[lemma]{Theorem}
\newtheorem{conjecture}[lemma]{Conjecture}

\newtheorem{prop}[lemma]{Proposition}

\newtheorem{claim}[lemma]{Claim}
\newtheorem{defi}[lemma]{Definition}

\theoremstyle{remark}

\newtheorem*{notation}{Notation}

\numberwithin{equation}{section}

\DeclareMathOperator{\tor}{{tor}}

\newcommand{\N}{{\mathbb N}}
\newcommand{\Q}{{\mathbb Q}}

\newcommand{\Qbar}{{\overline\Q}}

\newcommand{\C}{{\mathbb C}}
\newcommand{\Z}{{\mathbb Z}}
\newcommand{\bZ}{{\mathbb Z}}

\newcommand{\OO}{{\mathcal O}}
\newcommand{\tth}{^{\operatorname{th}}}
\newcommand{\lra}{\longrightarrow}
\newcommand{\A}{{\mathbb A}}

\DeclareMathOperator{\CL}{\sf{CL}}
\DeclareMathOperator{\CM}{\sf{C}}
\DeclareMathOperator{\Lab}{\sf{L}}

\newcommand{\bC}{{\mathbb C}}

\newcommand{\bA}{{\mathbb A}}

\newcommand{\cO}{\mathcal{O}}

\newcommand{\cF}{\mathcal{F}}

\begin{document}



\title[Endomorphisms of semiabelian varieties]{Mordell-Lang and Skolem-Mahler-Lech theorems for endomorphisms of semiabelian varieties}

\author{Dragos Ghioca}

\address{
Dragos Ghioca \\
Department of Mathematics \& Computer Science\\
University of Lethbridge \\
4401 University Drive \\ 
Lethbridge, Alberta T1K 3M4 
}

\email{ghioca@cs.uleth.ca}

\author{Thomas J. Tucker}

\address{
Thomas Tucker\\
Department of Mathematics\\
Hylan Building\\
University of Rochester\\
Rochester, NY 14627
}

\email{ttucker@math.rochester.edu}

\begin{abstract}
  Using the Skolem-Mahler-Lech theorem, we prove a dynamical
  Mordell-Lang conjecture for semiabelian varieties.
\end{abstract}

\date{\today} \subjclass{Primary 14G25; Secondary 37F10, 11C08}
\maketitle


\section{Introduction}

In 1991, Faltings \cite{Faltings} proved the Mordell-Lang conjecture.

\begin{thm}[Faltings]
\label{T:F}
Let $G$ be an abelian variety defined over the field of complex
numbers $\mathbb{C}$. Let $X\subset G$ be a subvariety and
$\Gamma\subset G(\mathbb{C})$ a finitely generated subgroup of
$G(\C)$. Then $X(\mathbb{C})\cap\Gamma$ is a finite union of cosets
of subgroups of $\Gamma$.
\end{thm}

Theorem~\ref{T:F} has been generalized to semiabelian varieties $G$ by
Vojta (see \cite{V1}) and to finite rank subgroups $\Gamma$ of $G$ by
McQuillan (see \cite{McQ}). Recall that a semiabelian variety (over
$\mathbb{C}$) is an extension of an abelian variety by a torus
$(\mathbb{G}_m)^k$.

Vojta's result implies that if $X$ is a subvariety of a semiabelian
variety $G$ defined over $\C$ and $X$ contains no translate of a
positive dimensional algebraic subgroup of $G$, then for any positive
integer $n$, the intersection of $X$ with the orbit of a point $P \in
G(\C)$ under the multiplication-by-$n$-map must be finite. In this
paper we describe the intersection of a subvariety of a semiabelian
variety $G$ defined over $\C$ with the orbit of a point $P \in G(\C)$
under any endomorphism $\phi:G\lra G$.

\begin{thm}
\label{semiabelian}
Let $G$ be a semiabelian variety defined over $\C$, and let $V\subset
G$ be a subvariety defined over $\C$. Let $\phi \in\ End(G)$, let
$P \in G(\C)$, and let $\OO:=\OO_{\phi}(P)$ be the orbit of $P$
under $\phi$. Then $V(\C)\cap\OO$ is either empty or a finite union of
orbits of the form $\OO_{\phi^N}(\phi^{\ell}(P))$, where
$N,\ell\in\N$.
\end{thm}

Prior to Vojta's proof of the semiabelian case of the Mordell-Lang
conjecture, Laurent \cite{Laurent} proved the Mordell-Lang conjecture
for any power of the multiplicative group. In particular. Laurent's
result shows that if $V\subset\mathbb{G}_m^k$ contains no translate of
a positive dimensional torus, then $V$ contains finitely many points
of the orbit of any point of $\bA^k$ under the map
$(X_1,\dots,X_k)\mapsto (X_1^{e_1},\dots,X_k^{e_k})$ (with
$e_i\in\mathbb{N}$) on $\bA^k$. This led the authors conjecture in
\cite{p-adic} that a similar result holds for any polynomial action on
the coordinates of the affine plane.  This conjecture is the
following.

\begin{conjecture}
\label{dynamical M-L}
Let $F_1,\dots,F_g$ be polynomials in $\bC[X]$, let $\cF$ be their
action coordinatewise on $\bA^g$, let $\cO_{\cF}((a_1,\dots,a_g))$
denote the $\cF$-orbit of the point $(a_1,\dots,a_g)\in \bA^g(\bC)$,
and let $V$ be a subvariety of $\bA^g$.  Then $V$ intersects
$\cO_{\cF}((a_1,\dots,a_g))$ in at most a finite union of orbits of
the form $\cO_{\cF^N}(\cF^{\ell}(a_1,\dots,a_g))$, for some
non-negative integers $N$ and $\ell$.
\end{conjecture}

Conjecture~\ref{dynamical M-L} fits into Zhang's far-reaching system of dynamical
conjectures \cite{ZhangLec}.  Zhang's conjectures include dynamical
analogues of the Manin-Mumford and Bogomolov conjectures for abelian
varieties (now theorems of Raynaud \cite{Raynaud1, Raynaud2}, Ullmo \cite{Ullmo},
and Zhang \cite{Zhang}), as well as a conjecture about the Zariski
density of orbits of points under fairly general maps from a
projective variety to itself.  The latter conjecture is related to our
Conjecture~\ref{dynamical M-L}, though neither conjecture contains the
other. Conjecture~\ref{dynamical M-L} has been proved in the case where $g=2$ and $V$ is a
line in $\bA^2$ (see \cite{Mike}). 

Bell \cite{Bell_S-M-L} proved Conjecture~\ref{dynamical M-L} in the
case the polynomials $F_i$ are all linear. More precisely, Bell proved
that if $\phi:\bA^g\to \bA^g$ is an automorphism, then for every
subvariety $V\subset \bA^g$ and any $P \in \bA^g$, the set of positive
integers $k$ for which $\varphi^k(P) \in V$ is either empty or equal
to a finite union of arithmetic progressions.  Bell's result is an
algebro-geometric generalization of a classical theorem by Skolem
\cite{Skolem} (which was later extended by Mahler \cite{Mahler-2} and
Lech \cite{Lech}).  The Skolem-Mahler-Lech theorem says that if
$\{u_k\}_{k\in\N}\in\C$ is a linear recurrence sequence, then the set
of all $k\in\N$ such that $u_k = 0$ is at most a finite union of
arithmetic progressions (some of them possibly constant). For a
quantitative version of the Skolem-Mahler-Lech theorem, we refer the
reader to \cite{ESS}.

The Skolem-Mahler-Lech theorem (see Proposition~\ref{ever 3}) is also
instrumental in our proof of Theorem~\ref{semiabelian}. For any
endomorphism $\phi$ of a semiabelian variety $G$ defined over $\C$,
the ring $\mathbb{Z}[\phi]$ is a finite extension of $\Z$; thus, a
cyclic $\Z[\phi]$-module is a finitely generated $\Z$-module.
Therefore, for each point $P\in G(\C)$, a subvariety $V$ of $G$
intersects the cyclic $\Z[\phi]$-module $\Gamma$ generated by $P$ in a
finite union of cosets of subgroups of $\Gamma$.  Since
$$V(\C)\cap \cO_{\phi}(P) = \left( V(\C)\cap \Gamma\right) \cap
\cO_{\phi}(P),$$
it suffices to describe the intersection of the
$\phi$-orbit $\cO_{\phi}(P)$ with each coset of a subgroup of
$\Gamma$, which is done using, among other techniques, also the Skolem-Mahler-Lech theorem. Our argument is similar to the one found in
\cite{Ghioca-TAMS} (see also \cite{Moosa-Scanlon} for the description
of the intersection of subvarieties of a semiabelian variety $G$
defined over a finite field $\mathbb{F}_q$ with $\Z[F]$-submodules of
$G$, where $F$ is the Frobenius on $\mathbb{F}_q$).  We note that our
methods are not $p$-adic (as in the case of the classical
Skolem-Mahler-Lech theorem), but rather geometric (as in the case of
the classical Mordell-Lang conjecture).  Thus, they represent a
connection of sorts between the Skolem-Mahler-Lech theorem and
Mordell-Lang conjecture.  It may also be possible to give a purely
$p$-adic analytic proof of Theorem~\ref{semiabelian}, using
logarithms, following the example of \cite{compositio, p-adic})

We briefly sketch the plan of our paper. In Section~\ref{second con}
we define the property for a general morphism from a variety into
itself to satisfy the ``Mordell-Lang condition'' (see
Definition~\ref{M-L condition}), and then we show the connection
between the Mordell-Lang condition and our Theorem~\ref{semiabelian},
and our Conjecture~\ref{dynamical M-L}. In Section~\ref{previous} we
state the Skolem-Mahler-Lech theorem, while in Section~\ref{semi-sub}
we prove Theorem~\ref{semiabelian}.

\begin{notation}
  Throughout this paper, $f^n$ denotes the $n\tth$ iterate of the
  map $f$.  We also use $\alpha^n$ and $X^n$ for the $n\tth$
  power of a constant or of $X$ itself, but this should not cause
  confusion.  We write $\N$ for the set of non-negative integers. An arithmetic progression in $\N$ is a set of the form $\{Nk+\ell\text{ : }k\in\N\}$ for some $N,\ell\in\N$ (if $N=0$, then the set consists of only one element $\{\ell\}$).   We
  write $\overline{K}$ for an algebraic closure of the field $K$.

If $\varphi: V \lra V$ is a map from a variety to itself and
$z$ is a point on $V$, we define the orbit
$\cO_{\varphi}(z)$ of $z$ under $\varphi$ as 
$$ \cO_{\varphi}(z) = \{ \varphi^k(z)\text{ : }k\in\N\}.$$
\end{notation}


\section{An equivalent conjecture}
\label{second con}

In this section, we present a condition that is equivalent to the
conclusions of Theorem~\ref{semiabelian} and of Conjecture~\ref{dynamical M-L} but has a statement that may
seem more familiar.

\begin{defi}
\label{M-L condition}
  Let $V$ be a variety over a field $L$ and let $\varphi: V \lra V$ be
  a morphism.  We say that $\varphi$ satisfies the {\bf
    Mordell-Lang condition} if for any subvariety $W$ of $V$ and any
  point $z$ in $V(L)$, there are $\varphi$-periodic subvarieties $Y_1,
  \dots, Y_m$ of $W$ and points $w_1, \dots, w_n$ in W such that
$$ \cO_{\varphi}(z) \cap W = \left( \bigcup_{i=1}^m (Y_i \cap  \cO_{\varphi}(z)) \right)
\bigcup \{ w_j\text{ : } 1\le j\le n\}.$$
\end{defi}

Faltings' proof of the Mordell-Lang conjecture and Vojta's extension say the following.
\begin{thm}
\label{C:F}
Let $G$ be a semiabelian variety defined over the field of complex
numbers $\mathbb{C}$. Let $X\subset G$ be a subvariety and
$\Gamma\subset G(\mathbb{C})$ a finitely generated subgroup of
$G(\C)$. Then there exist finitely many translates $(y_i+Y_i)\subset X$ of positive dimensional algebraic subgroups $Y_i\subset G$ (for $1\le i\le m$), and there exist finitely many points $x_j\in X$ (for $1\le j\le n$), such that
$$X(\mathbb{C})\cap\Gamma = \left( \bigcup_{i=1}^m (y_i+Y_i)\cap \Gamma\right) \bigcup \{x_j\text{ : }1\le j\le n\}.$$
\end{thm}
We also note that according to \cite[Lemme $10$]{Hindry-subvarieties}, if an irreducible subvariety $X$ of a semiabelian variety $G$ is periodic under the multiplication-by-$\ell$-map (for $\ell>1$), then $X$ is a translate of an algebraic subgroup of $G$.

In this paper we show that any endomorphism of a semiabelian variety over $\C$ satisfies the Mordell-Lang condition. We note that Bell \cite{Bell_S-M-L} proved that any automorphism of an affine variety also satisfies the Mordell-Lang condition. First we prove an equivalent formulation of the Mordell-Lang condition.

\begin{prop}
  A morphism $\varphi: V \lra V$ satisfies the Mordell-Lang
  condition if and only if for any subvariety $W$ of $V$ and any point
  $z \in V$, the intersection of $W$ with $\cO_{\varphi}(z)$ is equal
  to a finite union of orbits of the form
  $\cO_{\varphi^N}(\varphi^{\ell}(z))$, for some non-negative
  integers $N$ and $\ell$.
\end{prop}
\begin{proof}
Note that the proposition is trivial when $z$ is preperiodic under
$\varphi$.  Thus, we assume that $z$ is not preperiodic.  

Suppose that 
$$
\cO_{\varphi}(z) \cap W = \left( \bigcup_{i=1}^m (Y_i \cap
  \cO_{\varphi}(z)) \right) \bigcup  \{ w_j\text{ : }1\le j \le n\}.$$
Then for each $Y_i$, we let $N:=N_i$ be the period of $Y_i$ and for each $r\in\{0,\dots,N-1\}$, we
let $\ell:=\ell_{i,r}$ be the smallest non-negative integer $\ell\equiv r\text{ (mod $N$)}$ such that
$\varphi^{\ell}(z) \in Y_i$.  For each $w_j$, we let $N=0$ and let
$\ell$ be the unique non-negative integer such that $\varphi^\ell(z) =
w_j$.  Then we see that the intersection of $V$ with
$\cO_{\varphi}(z)$ is equal to a finite union of orbits of the form
$\cO_{\varphi^N}(\varphi^{\ell}(z))$, for some non-negative integers
$N$ and $\ell$.

Conversely, suppose that the intersection of $V$ with
$\cO_{\varphi}(z)$ is equal to a finite union of orbits of the form
$\cO_{\varphi^N}(\varphi^{\ell}(z))$, for some non-negative integers
$N$ and $\ell$.  For each orbit where $N=0$, we obtain a single point
$w_j$.  For each orbit where $N \not= 0$, taking the union of the
positive dimensional components of the Zariski closure
of the orbit yields a positive dimensional subvariety $Y_i$
of $W$ that is invariant under $\varphi^N$.  The zero-dimensional
components simply give additional points $w_j$.  
\end{proof}

Thus, Theorem~\ref{semiabelian} says that any endomorphism $\phi$ of a
semiabelian variety satisfies the Mordell-Lang condition. Furthermore,
our Conjecture~\ref{dynamical M-L} can be reformulated as follows.
\begin{conjecture}
\label{dynamical M-L 2}
Let $g\ge 1$, let $F_1,\dots,F_g$ be polynomials in $\bC[X]$, and let $\phi:\A^g\lra \A^g$ be the morphism 
$$\varphi(z_1,\dots,z_g) := (F_1(z_1),\dots,F_g(z_g)).$$
Then $\varphi$ satisfies the Mordell-Lang condition.
\end{conjecture}


\section{The Skolem-Mahler-Lech theorem}
\label{previous}

In this section we state the Skolem-Mahler-Lech theorem which will be
used in our proof of Theorem~\ref{semiabelian}. First we need to
introduce the basic set-up for \emph{linear recurrence sequences} (see
\cite{Evertse} for more details on linear recurrent sequences).
\begin{defi}
\label{recurrence sequences}
The sequence $\{u_k\}_{k\in\N}$ is a (linear) recurrence sequence, if there exists a positive integer $n$, and there exist constants $c_1,\dots,c_{n}$ (with $c_n\ne 0$) such that
\begin{equation}
\label{the relation}
u_{k+n} = \sum_{i=1}^{n} c_i u_{k+n-i}\text{, for each $k\in\N$.}
\end{equation}
\end{defi}

Assume $n$ is the smallest positive integer for which there exist constants
$c_1,\dots,c_n$ satisfying \eqref{the relation}. Every recurrence
sequence as above has a \emph{characteristic polynomial}
$$X^n - \sum_{i=1}^n c_i X^{n-i},$$
whose roots are called the \emph{characteristic roots} of $\{u_k\}_{k\in\N}$. Note that because $c_n\ne 0$, each characteristic root is nonzero. We let $\{\zeta_i\}_{i=1}^m$ be the distinct characteristic roots of $\{u_k\}_{k\in\N}$. Then there exist (single variable) polynomials $\{f_i\}_{i=1}^m$ such that for each $k\in \N$, we have
\begin{equation}
\label{recurrence formula}
u_k = \sum_{i=1}^m f_i(k) \zeta_i^k.
\end{equation}
If $K$ is an algebraically closed field, and $u_0,\dots,u_{n-1};c_1,\dots,c_n\in K$, then $\zeta_i\in K$ and $f_i\in K[X]$ for each $i$. Moreover, for \emph{any} given $\zeta_1,\dots,\zeta_m\in K$, and \emph{any} given polynomials $f_1,\dots,f_m\in K[X]$, the sequence $\{u_k\}_{k\in\N}\subset K$ defined by \eqref{recurrence formula} satisfies a linear recurrence relation.

The following result is the well-known Skolem-Mahler-Lech theorem (see \cite{ESS} for a more recent quantitative version).

\begin{prop}
\label{ever 3}
Let $m\in\N$, let $\zeta_1,\dots,\zeta_m\in\C^*$, and let $f_1,\dots,f_m\in \C[X]$. Then for every $C\in\C$, the set of all $k\in\N$ such that 
\begin{equation}
\label{second equation}
\sum_{i=1}^m f_i(k) \zeta_i^k = C
\end{equation} 
is either empty or a finite union of arithmetic progressions.
\end{prop}

\section{Semiabelian varieties with an endomorphism}
\label{semi-sub}

We are ready to prove Theorem~\ref{semiabelian}.  We begin with some
notation and some reductions.
  
  Since every endomorphism of a semiabelian variety is integral over
  $\bZ$, we may let $X^g - \sum_{i=1}^g e_i X^{g-i}$ be the minimal
  polynomial of $\phi$ over $\Z$. Then, for each $k\ge 0$, we
  have
\begin{equation}
\label{recurrence 2}
\phi^{k+g}(P) = \sum_{i=1}^g e_i \phi^{k+g-i}(P).
\end{equation}
If $e_g=0$, then we let $g_1$ be the largest index $i$ for which
$e_i\ne 0$. Because $\OO:=\OO_{\phi}(P)$ and $\OO_{\phi}(\phi^{g-g_1}(P))$ differ by
finitely many points, it suffices to prove Theorem~\ref{semiabelian}
for $\phi^{g-g_1}(P)$ instead of $P$. Thus, by replacing $P$ with
$\phi^{g-g_1}(P)$, we may replace $g$ by $g_1$ in \eqref{recurrence
  2}.  Hence, without loss of generality, we assume that the constant
$e_g$ in \eqref{recurrence 2} is nonzero.

For each $j\in\{0,\dots,g-1\}$ we define the sequence $\{z_{k,j}\}_{k\ge 0}$ as follows
\begin{equation}
\label{equation z 1}
z_{k,j} = 0 \text{ if }0\le k\le g-1\text{ and }k\ne j;
\end{equation}
\begin{equation}
\label{equation z 2}
z_{j,j} = 1\text{, and}
\end{equation}
\begin{equation}
\label{equation z 3}
z_{k,j} = \sum_{i=1}^{g} c_i z_{k-i,j}\text{ for all $k\ge g$.}
\end{equation}
Using \eqref{equation z 1} and \eqref{equation z 2} we obtain that
\begin{equation}
\label{equation z 4}
\phi^k(P) = \sum_{j=0}^{g-1} z_{k,j} \phi^j(P)\text{, for every $0\le k\le g-1$.}
\end{equation}
Using \eqref{recurrence 2}, \eqref{equation z 3} and \eqref{equation z 4}, an easy induction on $k$ shows that
\begin{equation}
\label{equation z 5}
\phi^k(P) = \sum_{j=0}^{g-1} z_{k,j}\phi^j(P)\text{, for every $k\ge 0$.}
\end{equation}
For each $j$, the sequence $\{z_{k,j}\}_{k\in\N}$ is a linear recurrence sequence; they all have the same characteristic polynomial. Hence there exists $m\in\N$, there exist $\{\gamma_i\}_{1\le i\le m}\subset\Qbar$, and there exist $\{f_{j,i}\}_{\substack{0\le j\le g-1 \\ 1\le i\le m}} \subset \Qbar[X]$ such that for every $j\in\{0,\dots,g-1\}$, and for every $k\in\N$, we have
\begin{equation}
\label{equation z gamma}
z_{k,j} = \sum_{i=1}^m f_{j,i}(k) \gamma_{i}^k.
\end{equation}
The numbers $\gamma_i$ are the characteristic roots of the recurrence sequences $\{z_{k,j}\}_{k\in\N}$; they are the same for each $j$.

Since $\phi$ is integral over $\bZ$, the module $\Z[\phi]$ is a finite
extension of $\Z$. Therefore, every finitely generated
$\Z[\phi]$-module is also a finitely generated $\Z$-module.  Let
$\Gamma$ be the cyclic $\Z[\phi]$-module generated by $P$. Then
$\Gamma$ is a finitely generated $\Z$-module, and so, by Vojta's proof
(\cite{V1}) of the Mordell-Lang conjecture for semiabelian varieties,
$V(\C)\cap\Gamma$ is a finite union of cosets $\{b_i+H_i\}_{i=1}^s$ of
subgroups $H_i\subset \Gamma$. Hence
\begin{equation}
\label{first intersection}
V(\C)\cap\OO = \bigcup_{i=1}^s \left( (b_i+H_i)\cap\OO\right).
\end{equation}
Thus, it suffices to show that for each coset 
\begin{equation}\label{b}
(b+H)\subset \Gamma
\end{equation}
appearing in \eqref{first intersection}, the intersection
$(b+H)\cap\OO$ is either empty or a finite union of orbits of the form
$\OO_{\phi^N}(\phi^{\ell}(P))$, where $N,\ell\in\N$.  Let us now fix
some notation.  

\begin{itemize}

\item We write 
$$ \Gamma = \Gamma_{\tor} \bigoplus \Gamma_1,$$
where $\Gamma_{\tor}$ is a finite torsion group and $\Gamma_1$ is a
free group of finite rank. 
\item $\{R_1,\dots,R_n\}$ is a
$\Z$-basis for $\Gamma_1$.
\item  For each $j\in\{0,\dots,g-1\}$, we let
$T^{(j)}\in\Gamma_{\tor}$ and $Q^{(j)}\in\Gamma_1$ such that
$\phi^j(P)=T^{(j)} + Q^{(j)}$. 
\item For each such $j$, we let
$\{a_{j,i}\}_{1\le i\le n}\subset\Z$ such that
\begin{equation}
\label{E:coefficient}
Q^{(j)} = \sum_{i=1}^n a_{j,i} R_i.
\end{equation}
Then, for each $k\in\N$, we have
\begin{equation}
\label{general form of the sum}
\phi^k(P) = \left(\sum_{j=0}^{g-1} z_{k,j} T^{(j)}\right) + \left(\sum_{j=0}^{g-1} z_{k,j} \sum_{i=1}^n a_{j,i} R_i\right).
\end{equation}

\item For $b$ in \eqref{b}, we write $b=b^{(0)}+b^{(1)}$, where
  $b^{(0)}\in\Gamma_{\tor}$, and $b^{(1)}\in \Gamma_1$.

\item We write $H_1:=H\cap\Gamma_1$, where $H$ is as in \eqref{b}. 
  
\item For each $h\in \Gamma_{\tor}$, if $(h+\Gamma_1)\cap H$ is not
  empty, we fix $(h+U_h)\in H$ for some $U_h\in \Gamma_1$.
  
\item For each $h\in \Gamma_{\tor}$, we let
  $$\OO^{(h)}:=\{\phi^k(P)\in\OO\text{ : }\left(-h
    -b^{(0)}+\phi^k(P)\right)\in\Gamma_1\}.$$
\end{itemize}

With the above notation, we have
\begin{equation}
\label{big intersection}
\begin{split}
  \OO\cap (b+H)
  & = \bigcup_{h\in \Gamma_{\tor}} \left( \OO^{(h)}\cap
    \left((h+b^{(0)})+(b^{(1)} + U_h+H_1)\right) \right)\\
  = \bigcup_{h\in \Gamma_{\tor}} &  \left( (h+b^{(0)})+\left(\left(
        -h-b^{(0)}+\OO^{(h)}\right)\cap (b^{(1)}+U_h+H_1)\right)
  \right),
\end{split}
\end{equation}
where $-x+Y:=\{-x+y\text{ : }y\in Y\}$ for every point $x$, and every subset $Y$ of $G$. For each $h\in\Gamma_{\tor}$ such that $(h+\Gamma_1)\cap H=\emptyset$ the above intersection is empty (and there is no $U_h$).
 
Using \eqref{general form of the sum} and \eqref{big intersection}, we conclude that $\OO\cap(b+H)$ is a finite union over $h\in \Gamma_{\tor}$ of the points $\phi^k(P)$ corresponding to $k\in\N$ such that
\begin{equation}
\label{condition}
\sum_{j=0}^{g-1} z_{k,j} T^{(j)} = h+b^{(0)}\text{ and
}\sum_{j=0}^{g-1} \big( z_{k,j} \sum_{i=1}^n a_{j,i} R_i \big) \in (b^{(1)}+U_h +H_1).
\end{equation}
\begin{lemma}
\label{torsion}
Let $h\in \Gamma_{\tor}$. The set of $k\in\N$ such that $\sum_{j=0}^{g-1} z_{k,j} T^{(j)} = h$ is either empty or a finite union of arithmetic progressions.
\end{lemma}

\begin{proof}[Proof of Lemma~\ref{torsion}.]
  Choose $N\in\mathbb{N}$ such that $\Gamma_{\tor}\subset \Gamma[N]$.
  Then the value of $\sum_{j=0}^{g-1} z_{k,j} T^{(j)}$ is
  completely determined by the values of the $z_{k,j}$ modulo $N$.
  Since there are finitely many $g$-tuples of integers modulo $N$, and
  each $\{z_{k,j}\}_{k\in\N}$ is a linear recurrence sequence of
  degree $g$ in $\mathbb{Z}$, it follows that each sequence
  $\{z_{k,j}\}_{k\in\N}$ eventually begins to repeat itself modulo $N$, i.e.
  each sequence is preperiodic modulo $N$.  Thus, each value taken by
  $\sum_{j=0}^{g-1}z_{k,j}T^{(j)}$ is attained for $k\in\N$ living in
  a finite union of arithmetic progressions.
\end{proof}

We will now prove a more difficult Lemma from which the proof of
Theorem~\ref{semiabelian} will follow easily.
\begin{lemma}
\label{C:non-torsion}
Let $h\in \Gamma_{\tor}$ be fixed such that $(h+\Gamma_1)\cap H$ is not empty. The set of all $k\in\N$ for which 
\begin{equation}
\label{E:nontorsion}
\sum_{j=0}^{g-1} \big( z_{k,j}  \sum_{i=1}^n a_{j,i} R_i \big) \in (b^{(1)}+U_h +H_1)
\end{equation} 
is either empty or a finite union of arithmetic progressions.
\end{lemma}

\begin{proof}[Proof of Lemma~\ref{C:non-torsion}.]
We first define three classes of subsets of $\mathbb{Z}^n$.
\begin{defi}
A $\CM$-subset of $\mathbb{Z}^n$ is a set $C(d_1,\dots,d_n,D_1,D_2)$, where $d_1,\dots,d_n,D_1,D_2\in\mathbb{Z}$ and $D_2\ne 0$, containing all solutions $(x_1,\dots,x_n)\in\mathbb{Z}^n$ of $\sum_{i=1}^n d_ix_i\equiv D_1\text{ (mod $D_2$)}$.

An $\Lab$-subset of $\mathbb{Z}^n$ is a set $L(d_1,\dots,d_n,D)$, where $d_1,\dots,d_n,D\in\mathbb{Z}$, containing all solutions $(x_1,\dots,x_n)\in\mathbb{Z}^n$ of 
$\sum_{i=1}^n d_ix_i=D$.

A $\CL$-subset of $\mathbb{Z}^n$ is either a $\CM$-subset or an $\Lab$-subset of $\mathbb{Z}^n$.
\end{defi}
The $\CM$-subsets may be thought of as satisfying congruence
relations, while the $\Lab$-subsets satisfy linear conditions.

\begin{claim}
\label{C:non-torsion-4}
There exist $\CL$-subsets $S_1,\dots,S_n$ of $\mathbb{Z}^n$ such that a point $R:=\sum_{i=1}^n c^{(i)}R_i$ lies in $U_h+b^{(1)}+ H_1$ if and only if 
$$(c^{(1)},\dots,c^{(n)})\in\bigcap_{i=1}^n S_i.$$
\end{claim}
\begin{proof}[Proof of Claim~\ref{C:non-torsion-4}.]
Because $H_1\subset \Gamma_1$ and $\Gamma_1$ is a free $\mathbb{Z}$-module with basis $\{R_1,\dots,R_n\}$, we can find (after a possible relabeling of $R_1,\dots,R_n$) a $\mathbb{Z}$-basis $Q_1,\dots,Q_{\ell}$ (with $1\le \ell\le n$) of $H_1$ of the following form:
$$Q_1=\beta_{1}^{(i_1)}R_{i_1}+\dots+\beta_{1}^{(n)}R_n;$$
$$Q_2=\beta_{2}^{(i_2)}R_{i_2}+\dots+\beta_{2}^{(n)}R_n;$$
and in general
\begin{equation}\label{Q}
Q_j = \beta_{j}^{(i_j)}R_{i_j} + \dots + \beta_{j}^{(n)}R_n 
\end{equation}
for each $j\le \ell$, where $$1\le i_1<i_2<\dots<i_{\ell}\le n$$
and
all $\beta_{j}^{(i)}\in\mathbb{Z}$. We also assume
$\beta_j^{(i_j)}\ne 0$ for every $j\in\{1,\dots,\ell\}$.

Let $b_{1,1},\dots,b_{1,n}\in\mathbb{Z}$ such that $b^{(1)}+ U_h=\sum_{j=1}^n b_{1,j}R_j$. Then $R\in (b^{(1)}+U_h+H_1)$ if and only if there exist integers $k_1,\dots,k_{\ell}$ such that 
\begin{equation}
\label{E:v1}
R=b^{(1)}+U_h+\sum_{i=1}^{\ell}k_i Q_i.
\end{equation}
Using the expressions for the $Q_i$ (in \eqref{Q}), $(b^{(1)}+U_h)$,
and $R$ in terms of the $\mathbb{Z}$-basis $\{R_1,\dots,R_n\}$ of
$\Gamma_1$, we obtain the following relations for the coefficients
$c^{(j)}$:
\begin{equation}
\label{E:a}
c^{(j)}=b_{1,j}\text{ for every $1\le j<i_1$;}
\end{equation}
\begin{equation}
\label{E:v2}
c^{(j)}=b_{1,j}+k_1\beta_1^{(j)}\text{ for every $i_1\le j<i_2$;}
\end{equation}
\begin{equation}
\label{E:v3}
c^{(j)}=b_{1,j}+k_1\beta_1^{(j)}+k_2\beta_2^{(j)}\text{ for every $i_2\le j<i_3$}
\end{equation}
and so on, until
\begin{equation}
\label{E:v4}
c^{(n)}=b_{1,n}+\sum_{i=1}^{\ell} k_i\beta_i^{(n)}.
\end{equation}

We interpret the above relations as follows: the numbers
$c^{(1)},\dots,c^{(n)}$ are the unknowns, while the numbers
$k_1,\dots,k_{\ell}$ are integer parameters, and all $b_{1,i}$ and
$\beta_j^{(i)}$ are integer constants. We will show, by eliminating
the parameters $k_i$, that the unknowns $c^{(j)}$ must satisfy $\ell$
linear congruences and $(n-\ell)$ linear equations with coefficients
involving only the $b_{1,i}$ and the $\beta^{(i)}_j$. Each such
equation will generate a $\CL$-set.

We begin by expressing equation \eqref{E:v2} for $j=i_1$ as a linear
congruence modulo $\beta_1^{(i_1)}$ and obtain
\begin{equation}
\label{E:b}
c^{(i_1)}\equiv b_{1,i_1}\big(\text{ mod $\beta_1^{(i_1)}$} \big).
\end{equation}
Equation \eqref{E:v2} also gives us
$k_1=\frac{c^{(i_1)}-b_{1,i_1}}{\beta_1^{(i_1)}}$.  Substituting this
formula for $k_1$ into \eqref{E:v2} for each $i_1<j<i_2$, we obtain
\begin{equation}
\label{E:c}
c^{(j)}=b_{1,j}+\frac{c^{(i_1)}-b_{1,i_1}}{\beta_1^{(i_1)}}\beta_1^{(j)}\text{ for every $i_1<j<i_2$.}
\end{equation}
We then express \eqref{E:v3} for $j=i_2$ as a linear congruence modulo $\beta_2^{(i_2)}$ (also using the expression for $k_1$ computed above). We obtain
\begin{equation}
\label{E:d}
c^{(i_2)}\equiv b_{1,i_2} +
\frac{c^{(i_1)}-b_{1,i_1}}{\beta_1^{(i_1)}}
\beta_1^{(i_2)}\left(\text{ mod $\beta_2^{(i_2)}$} \right).
\end{equation}
Next we solve for $k_2$ using \eqref{E:v3} for $j=i_2$ (along with our
formula above for $k_1$) and obtain
$$k_2=\frac{c^{(i_2)}-b_{1,i_2}-\frac{c^{(i_1)}-b_{1,i_1}}{\beta_1^{(i_1)}}\beta_1^{(i_2)}}{\beta_2^{(i_2)}}.$$
Then we substitute this formula for $k_2$ in \eqref{E:v3} for
$i_2<j<i_3$ and obtain
\begin{equation}
\label{E:z}
c^{(j)}=b_{1,j}+\frac{c^{(i_1)}-b_{1,i_1}}{\beta_1^{(i_1)}}\cdot\beta_1^{(j)}+\frac{c^{(i_2)}-b_{1,i_2}-\frac{c^{(i_1)}-b_{1,i_1}}{\beta_1^{(i_1)}}\beta_1^{(i_2)}}{\beta_2^{(i_2)}}\cdot\beta_2^{(j)}.
\end{equation}
Continuing onward in this manner, we express $c^{(n)}$ in terms of the
integers $b_{1,i_1},\dots,b_{1,i_{\ell}},b_{1,n}$,
$c^{(i_1)},\dots,c^{(i_{\ell})}$, and $\{\beta_j^{(i)}\}_{j,i}$.

We observe that all of the above congruences and linear equations can
be written as linear congruences or linear equations over $\mathbb{Z}$
(after clearing the denominators). For example, the congruence
equation \eqref{E:d} can be written as the following linear congruence
over $\mathbb{Z}$:
$$\beta_1^{(i_1)}\cdot c^{(i_2)}\equiv
\beta_1^{(i_1)}b_{1,i_2}+\left(c^{(i_1)}-b_{1,i_1}\right)\beta_1^{(i_2)}\left(\text{
    mod $\beta_1^{(i_1)}\cdot\beta_2^{(i_2)}$}\right).$$
Hence all the
above conditions that must be satisfied by $c^{(j)}$ for which
$$\sum_{j=1}^nc^{(j)}R_j\in (b^{(1)}+U_h+H_1)$$
are either linear
equations over $\mathbb{Z}$ (giving rise to $\Lab$-subsets) or linear congruences
over $\mathbb{Z}$ (giving rise to $\CM$-subsets). There are
precisely $\ell$ congruences (corresponding to the $\ell$ degrees of
freedom introduced by the parameters $k_i$) and $(n-\ell)$ linear
equations. This concludes the proof of Claim~\ref{C:non-torsion-4}.
\end{proof}

We will now show that for each $S_i$ that appears in
Claim~\ref{C:non-torsion-4}, there exists at most finitely many
arithmetic progressions $W^{(i)}_j\subset\mathbb{N}$ such that
$k\in\bigcup_j W^{(i)}_j$ if and only if $(c^{(1)},\dots,c^{(n)})\in
S_i$, where
$$\sum_{i=1}^n c^{(i)}R_i := \sum_{j=0}^{g-1} \big( z_{k,j}
\sum_{i=1}^n a_{j,i} R_i \big) .$$
This will show that there exists at
most a finite number of arithmetic progressions
$$\widetilde{W}:=\bigcap_{i=1}^n\left(\bigcup_j
  W^{(i)}_j\right)\subset\mathbb{N}$$
such that $k\in\widetilde{W}$ if
and only if
\begin{equation}
\label{E:repeat non-torsion 2}
\sum_{j=0}^{g-1} \big( z_{k,j}  \sum_{i=1}^n a_{j,i} \big) R_i \in\left(b^{(1)}+ U_h+ H_1\right).
\end{equation}

\begin{claim}
\label{C:non-torsion-2}
Let $C:=C(d_1,\dots,d_n,D_1,D_2)$ be a $\CM$-subset of $\mathbb{Z}^n$. There exists at most a finite number of arithmetic progressions $W_j\subset\mathbb{N}$ such that $k\in\bigcup_j W_j$ if and only if $(c^{(1)},\dots,c^{(n)})\in C$, where
\begin{equation}
\label{E:coefficient 20}
\sum_{i=1}^n c^{(i)} R_i := \sum_{j=0}^{g-1} \big( z_{k,j} \sum_{i=1}^n a_{j,i} R_i \big)
.
\end{equation}
\end{claim}

\begin{proof}[Proof of Claim~\ref{C:non-torsion-2}.]
Using \eqref{E:coefficient 20}, we conclude that for every $1\le i\le
n$, we have
\begin{equation}
\label{E:C-coefficient}
c^{(i)}=\sum_{j=0}^{g-1} a_{j,i}z_{k,j}.
\end{equation}
Hence, the congruence equation $\sum_{i=1}^n d_ic^{(i)}\equiv D_1\text{ (mod $D_2$)}$ yields the congruence 
\begin{equation}
\label{E:congruence}
\sum_{j=0}^{g-1} h_{j}z_{k,j}\equiv D_1\text{ (mod $D_2$)}
\end{equation}
for integers $h_{j}:=\sum_{i=1}^n d_i a_{j,i}$, for each ${0\le
  j\le g-1}$ (we recall that all $a_{j,i}\in\mathbb{Z}$).  As noted in
the proof of Lemma \ref{torsion}, recursively defined sequences over
$\Z$, such as $\{z_{k,j}\}_{k\in\N}$, are preperiodic modulo any
nonzero integer (hence, they are preperiodic modulo $D_2$). Therefore
the set of all solutions $k\in\mathbb{N}$ to \eqref{E:congruence} is at most a
finite union $\bigcup_j W_j$ of arithmetic progressions in
$\mathbb{N}$.
\end{proof}

\begin{claim}
\label{C:non-torsion-3}
Let $L:=L(d_1,\dots,d_n,D)$ be an $\Lab$-subset of $\mathbb{Z}^n$.
There exist at most finitely many arithmetic progressions
$W_j\subset\mathbb{N}$ such that $k\in\bigcup_j W_j$ if and only if
$(c^{(1)},\dots,c^{(n)})\in L$, where
\begin{equation}
\label{E:coefficient 21}
\sum_{i=1}^n c^{(i)} R_i :=
\sum_{j=0}^{g-1} \big( z_{k,j} \sum_{i=1}^n a_{j,i} R_i \big) .
\end{equation}
\end{claim}

\begin{proof}[Proof of Claim~\ref{C:non-torsion-3}.]
  Using \eqref{E:coefficient 21} and \eqref{equation z gamma}, we
  conclude that for every $1\le i\le n$, we have
\begin{equation}
\label{E:tipical}
c^{(i)}=\sum_{j=0}^{g-1} a_{j,i}\sum_{\ell=1}^{m} f_{j,\ell}(k)\gamma_{\ell}^{k}.
\end{equation}
The linear equation $\sum_{i=1}^n d_i c^{(i)}=D$ yields the following
equation (after collecting the coefficients of $\gamma_{\ell}^{k}$ for
each $1\le \ell\le m$):
\begin{equation}
\label{E:M-L}
\sum_{\ell=1}^m  f_{\ell}(k) \gamma_{\ell}^{k}=D,
\end{equation}
where $f_{\ell}:=\sum_{j=0}^{g-1} \sum_{i=1}^n d_i a_{j,i}\cdot
f_{j,\ell} \in\Qbar[X]$ for each $\ell\in\{1,\dots,m\}$. Using
Proposition~\ref{ever 3}, the set of all $k\in\N$ satisfying
\eqref{E:M-L} is at most a finite union of arithmetic progressions, as
desired.
\end{proof}

Since the intersection of two arithmetic progressions in $\mathbb{N}$
is another arithmetic progression (or the empty set), Claims~\ref{C:non-torsion-2} and
\ref{C:non-torsion-3} finish the proof of Lemma~\ref{C:non-torsion}.
\end{proof}

We are now ready to complete the proof of Theorem~\ref{semiabelian}.
\begin{proof}[Proof of Theorem~\ref{semiabelian}.]
  It follows from Lemmas~\ref{torsion} and \ref{C:non-torsion}, that
  for each fixed $h\in \Gamma_{\tor}$, there is at most a finite union $W_h$ of
  arithmetic progressions in $\mathbb{N}$ such that $k\in\mathbb{N}$
  satisfies the equations
\begin{equation}
\label{E:repeat torsion}
\sum_{j=0}^{g-1}z_{k,j}T^{(j)} = h+b^{(0)} \quad \text{and}
\end{equation}
\begin{equation}
\label{E:repeat non-torsion}
\sum_{j=0}^{g-1} \big(z_{k,j} \sum_{i=1}^n a_{j,i} R_i \big)  \in
\left(b^{(1)} + U_h + H_1\right)
\end{equation}
if and only if $k\in W_h$. Using \eqref{big intersection},
\eqref{condition}, and that $\Gamma_{\tor}$ is finite, we conclude that
the set of all $k\in\N$ for which $\phi^k(P) \in (b+H)$ is either
empty or a finite union of arithmetic progressions.
\end{proof}

\def\cprime{$'$} \def\cprime{$'$} \def\cprime{$'$} \def\cprime{$'$}
\providecommand{\bysame}{\leavevmode\hbox to3em{\hrulefill}\thinspace}
\providecommand{\MR}{\relax\ifhmode\unskip\space\fi MR }
\providecommand{\MRhref}[2]{%
  \href{http://www.ams.org/mathscinet-getitem?mr=#1}{#2}
}
\providecommand{\href}[2]{#2}


\end{document}